\documentclass[reqno]{amsart}

\input xy
\xyoption{all}
\usepackage{epsfig}
\usepackage{color}
\usepackage{amsthm}
\usepackage{amssymb}
\usepackage{amsmath}
\usepackage{amscd}
\usepackage{amsopn}
\usepackage{url}
\usepackage{hyperref}\hypersetup{colorlinks}

\usepackage{color} 

\definecolor{darkred}{rgb}{1,0,0} 
\definecolor{darkgreen}{rgb}{0,0.8,0}
\definecolor{darkblue}{rgb}{0,0,1}

\hypersetup{colorlinks,
linkcolor=darkblue,
filecolor=darkgreen,
urlcolor=darkred,
citecolor=darkgreen}

\newcommand{\labell}[1] {\label{#1}}

\numberwithin{equation}{section}
\newtheorem {Theorem}{Theorem}
\numberwithin{Theorem}{section}

\newtheorem {Lemma}[Theorem]    {Lemma}

\newtheorem {Proposition}[Theorem]{Proposition}
\newtheorem {Corollary}[Theorem]{Corollary}
\theoremstyle{definition}
\newtheorem{Definition}[Theorem]{Definition}
\theoremstyle{remark}
\newtheorem{Remark}[Theorem]{Remark}
\newtheorem{Example}[Theorem]{Example}

\expandafter\chardef\csname pre amssym.def at\endcsname=\the\catcode`\@
\catcode`\@=11
\def\undefine#1{\let#1\undefined}
\def\newsymbol#1#2#3#4#5{\let\next@\relax
 \ifnum#2=\@ne\let\next@\msafam@\else
 \ifnum#2=\tw@\let\next@\msbfam@\fi\fi
 \mathchardef#1="#3\next@#4#5}
\def\mathhexbox@#1#2#3{\relax
 \ifmmode\mathpalette{}{\m@th\mathchar"#1#2#3}%
 \else\leavevmode\hbox{$\m@th\mathchar"#1#2#3$}\fi}
\def\hexnumber@#1{\ifcase#1 0\or 1\or 2\or 3\or 4\or 5\or 6\or 7\or 8\or
 9\or A\or B\or C\or D\or E\or F\fi}

\font\teneufm=eufm10
\font\seveneufm=eufm7
\font\fiveeufm=eufm5
\newfam\eufmfam
\textfont\eufmfam=\teneufm
\scriptfont\eufmfam=\seveneufm
\scriptscriptfont\eufmfam=\fiveeufm

\catcode`\@=\csname pre amssym.def at\endcsname

\def    \eps    {\epsilon}
\def    \ka    {\kappa}
\def    \nat    {{\natural}}

\newcommand{\id}{{\mathit id}}

\newcommand{\const}{{\mathit const}}

\def    \C      {{\mathbb C}}
\def    \R      {{\mathbb R}}

\def    \Z      {{\mathbb Z}}
\def    \N      {{\mathbb N}}
\def    \Q      {{\mathbb Q}}

\def    \T      {{\mathbb T}}

\def    \RP     {{\mathbb R}{\mathbb P}}

\def    \12    {{\frac{1}{2}}}

\def    \p      {\partial}

\def    \d  {\mathrm{d}}

\def    \U     {\operatorname{U}}

\def    \HF     {\operatorname{HF}}
\def    \HC     {\operatorname{HC}}
\def    \CC     {\operatorname{CC}}

\def    \H     {\operatorname{H}}

\def    \HM     {\operatorname{HM}}
\def    \MC    {\operatorname{CM}}

\def    \MUCZ  {\operatorname{\mu_{\scriptscriptstyle{CZ}}}}

\begin{document}


\setlength{\smallskipamount}{6pt}
\setlength{\medskipamount}{10pt}
\setlength{\bigskipamount}{16pt}





\title[Closed Reeb orbits on the sphere]{Closed Reeb orbits on the sphere and
  symplectically degenerate maxima}

\author[Ginzburg]{Viktor L. Ginzburg}
\author[Hein]{Doris Hein}
\author[Hryniewicz]{Umberto L. Hryniewicz}
\author[Macarini]{Leonardo Macarini}

\address{VG: Department of Mathematics, UC Santa Cruz, Santa Cruz, CA
  95064, USA} \email{ginzburg@ucsc.edu}

\address{DH: School of Mathematics Institute for Advanced Study,
  Princeton, NJ 08540, USA} \email{dhein@math.ias.edu}

\address{UH and LM: Universidade Federal do Rio de Janeiro, Instituto de
  Matem\'atica, Cidade Universit\'aria, CEP 21941-909, Rio de Janeiro,
  Brazil} \email{umberto@labma.ufrj.br} \email{leonardo@impa.br}


\date{\today} 

\thanks{The work is partially supported by the NSF grants DMS-0906204 (VG),
  DMS-1128155 (DH), DMS-0635607 (UH), and by CNPq (LM)}


\begin{abstract} We show that the existence of one simple closed Reeb
  orbit of a particular type (a symplectically degenerate
  maximum) forces the Reeb flow to have infinitely many periodic
  orbits. We use this result to give a different proof
  of a recent theorem of Cristofaro-Gardiner and Hutchings asserting
  that every Reeb flow on the standard contact three-sphere has at least two
  periodic orbits. Our methods are based on adapting the machinery
  originally developed for proving the Hamiltonian Conley conjecture
  to the contact setting.
 
\end{abstract}

\maketitle

\tableofcontents

\section{Introduction and main results}
\label{sec:intro}

\subsection{Introduction}
In this paper, we use the techniques originally developed for
proving the Hamiltonian Conley conjecture to show that the existence
of one simple closed Reeb orbit of a particular type -- a 
symplectically degenerate maximum (SDM) -- forces the Reeb flow to have
infinitely many periodic orbits. We apply this result to give a
different proof of a recent theorem of Cristofaro-Gardiner and
Hutchings, \cite{CGH}, asserting that every contact form supporting 
the standard contact structure on the three-sphere has at least two closed
Reeb orbits.

The proof of our Conley conjecture type result closely follows its
Hamiltonian counterpart from \cite{GG:gaps} (see also
\cite{Gi:CC,He:CC,Hi,SZ} for other relevant results) with local
contact homology, introduced in \cite{HM}, used in place of local
Floer homology.

The idea connecting the Conley conjecture and the existence of at
least two closed Reeb orbits on $S^3$ is, of course, that if there were a
contact form (giving rise to the standard contact structure)
with only one simple closed Reeb orbit, this orbit would be an SDM,
which in turn would ensure the existence of infinitely many periodic
orbits. This is established by a straightforward index analysis with two
non-trivial inputs. One comes from the results in \cite{HWZ1,HWZ2}
eliminating certain index patterns. The other one is a technical
theorem relating the local contact homology of an isolated
($\ka$-iterated) closed Reeb orbit with the ($\Z_\ka$-equivariant)
local Floer homology of its Poincar\'e return map (cf.\
\cite{Gi:CC,HM}). This is Theorem \ref{thm:ch-vs-fh}, which enables us
to apply the results from \cite{GG:gap} concerning local Floer homology in
the contact context, and which we feel may be of use for a variety of
other questions.

\subsection{Main results}
\labell{sec:main-results}
Let us now state the two principal results of the paper.

\begin{Theorem}[\cite{CGH}]
\label{thm:S^3}
Let $\alpha$ be a contact form on $S^3$ such that $\ker\alpha$ is the
standard contact structure. Then the Reeb flow of $\alpha$ has at
least two closed orbits.
\end{Theorem}

To put this theorem in context, let us recall some relevant results.
First, note that the existence of at least one closed Reeb orbit in
this case is a theorem of Viterbo; \cite{Vi:WC}.  Furthermore, any
strictly convex hypersurface in $\R^4$ carries either two or
infinitely many closed characteristics; \cite{HWZ:98}. The same holds
for any non-degenerate contact form on $S^3$ supporting the standard
contact structure, provided that all stable and unstable manifolds of
the hyperbolic periodic orbits intersect transversally; \cite{HWZ:03}.
In fact, any non-degenerate contact form on a closed three-manifold has
at least two simple closed Reeb orbits, and, if the manifold is not a
lens space, there are at least three such orbits; \cite{HT}. Moreover,
any symmetric, compact star-shaped hypersurface in $\R^{4}$ has at
least two closed characteristics regardless of whether the resulting contact
form is non-degenerate or not; \cite{Lo:le}. Finally, we refer the
reader to, e.g., \cite{EH,HZ,Lo,Wa} and references therein for results
and conjectures concerning closed Reeb orbits in higher dimensions.

As has been pointed out above, we establish Theorem \ref{thm:S^3} as a
consequence of the following

\begin{Theorem}
\label{thm:sdm}
Let $(M^{2n-1},\ker\alpha)$ be a contact manifold admitting a strong
symplectically aspherical filling $(W,\omega)$. Assume that the Reeb
flow of $\alpha$ has an isolated simple closed Reeb orbit which is a
symplectically degenerate maximum (SDM) with respect to some
symplectic trivialization of $\ker\alpha\!\mid_x$. Furthermore, when
$x$ is not contractible in $W$, we require $(W,\omega)$ to be
symplectically atoroidal. Then the Reeb flow of $\alpha$ has infinitely
many periodic orbits.
\end{Theorem}

Recall that $(W,\omega)$ is called symplectically aspherical when
$[\omega]\!\mid_{\pi_2(W)}=0$ and $c_1(TW)\!\mid_{\pi_2(W)}=0$, and
$(W,\omega)$ is symplectically atoroidal when $[\omega]$ and $c_1(TW)$
vanish on all toroidal classes, i.e., $\left<[\omega],v\right>=0$ and
$\left<c_1(TW),v\right>=0$ for every map $v\colon
\T^2\to W$. We refer the reader to Section
\ref{sec:sdm} for the definition and detailed discussion of
symplectically degenerate maxima. The proof of Theorem \ref{thm:sdm}
draws heavily from recent work on the Conley conjecture in the
Hamiltonian setting; see \cite{Gi:CC,GG:gaps,He:CC,Hi}.  In
particular, similarly to the Hamiltonian case, Theorem \ref{thm:sdm}
follows from a result (Theorem \ref{thm:ch-sdm}) asserting that the
contact homology of $\alpha$ does not vanish for certain action
intervals.  It is very likely that the requirement that $W$ is
symplectically aspherical can be relaxed, with minimal modifications
to the proof, and replaced by the assumption that
$c_1(TW)\!\mid_{\pi_2(W)}=0$ as in \cite{GG:gaps,He:CC} when $x$ is
contractible in $W$ or, in general, by the assumption that $c_1(TW)$
is atoroidal.  Let us also call the reader's attention to the
condition that the SDM orbit is simple. Although this assumption
appears to be of a technical nature (cf.\ \cite{Hi93,Hi97}), it does
play a crucial role in our proof and eliminating it would probably
require developing a different approach to the problem.

As an immediate application of Theorem \ref{thm:sdm}, we observe that
a Riemannian or Finsler manifold has infinitely many closed geodesics
whenever it has one prime closed SDM geodesic. This is the case, for
instance, for the closed geodesic considered in \cite[Proposition
I]{Hi93} whenever this geodesic is prime (and isolated), and hence
the proposition follows, under these hypotheses, from our Theorem
\ref{thm:sdm}. Furthermore, note that Theorem \ref{thm:sdm} completes
a ``mostly Floer--theoretic'' proof of the existence of infinitely
many closed geodesics on $S^2$; see \cite{Ba,Fr} and also
\cite{Hi93,Hi97} and references therein for the original
argument. Indeed, as is observed in \cite{Hi97}, the curve shortening
method from \cite{Gr} yields a Lusternik--Schnirelmann closed geodesic
$x$ with non-trivial local (Morse) homology in degree three, which
without loss of generality can be assumed to be isolated. If $x$ is
non-rotating (i.e., its mean index is equal to two), $x$ is an SDM and
Theorem \ref{thm:sdm} applies. When $x$ is not non-rotating, the
existence of infinitely many closed geodesics is a consequence of the
main theorem in \cite{HMS}. (See also \cite{BH,CKRTZ,Ke} for other
relevant results including a symplectic proof of Franks' theorem.)

\begin{Remark}
\label{rmk:min0}
As is clear from the proof, a result similar to Theorem \ref{thm:sdm}
holds when an SDM orbit is replaced by a symplectically degenerate
minimum (SDMin); see Remark \ref{rmk:min}. For instance, the closed
geodesic $x$ from the main theorem in \cite{Hi97} is an SDMin,
whenever $x$ is isolated. Hence, in particular, combining Theorems
\ref{thm:sdm} and \ref{thm:ch-sdm} with Remark \ref{rmk:min2}, we
reprove the main result of \cite{Hi97} under the additional
assumption that the geodesic in question is isolated and prime.
\end{Remark}

A word is due on the degree of rigor in this paper, which varies
considerably between its different parts. First of all, it should be
noted that the paper heavily relies on the machinery of contact
homology (see, e.g., \cite{Bou,SFT} and references therein), which is
yet to be fully put on a rigorous basis (see \cite{HWZ3,HWZ4}). With
this reservation in mind, the proof of Theorem \ref{thm:sdm} is
essentially complete. Here the usage of contact homology is rather
formal and the argument follows closely the proof of the Conley
conjecture as given in \cite{GG:gaps}, omitting for the sake of
brevity some details which are straightforward to fill in. Theorem
\ref{thm:S^3} is rigorously established as a consequence of Theorem
\ref{thm:sdm} and, at least on the expository level, of a technical
result (Theorem \ref{thm:ch-vs-fh}), mentioned above, relating local
Floer and local contact homology. The proof of Theorem
\ref{thm:ch-vs-fh} is only sketched in this paper. Although some of
the details can be filled in as in \cite{HM}, the argument must
ultimately utilize the machinery of multivalued perturbations. The
difficulty arising here is illustrated in Section \ref{sec:Morse},
where we analyze the Morse theoretic counterpart of the problem by
discussing an isomorphism between equivariant and invariant Morse
homology for manifolds equipped with a finite group action.

However, it is worth pointing out that, in fact, the proof of Theorem
\ref{thm:S^3} depends only on a particular case of Theorem
\ref{thm:ch-vs-fh} and on some results from \cite{HWZ1,HWZ2}. This
particular case is Corollary \ref{cor:rnak}, originally proved in
\cite{HM} and ultimately independent of the theorem. (An alternative
approach is also outlined in Section~\ref{sec:S^3}).

Finally, note that the foundational difficulties discussed above can be
circumvented by using equivariant symplectic homology (see
\cite{BO:gysin,BO:transv,Vi}) in place of contact homology; cf.\
\cite{HM} vs.\ \cite{McL}. Alternatively, one may utilize equivariant
Rabinowitz Floer homology, which appears to carry essentially the same
information as contact or equivariant symplectic homology; see
\cite{AF,CFO}. However, in either case, this approach is likely to require
geometrically much less transparent, perhaps cumbersome, arguments
than those given in the present paper.

\medskip
\noindent\textbf{Notation and conventions.} 
Throughout the paper, we use the conventions and notation of
\cite{GG:gap} when working with notions from Hamiltonian dynamics,
including Floer homology and the Conley--Zehnder index, and of
\cite{HM} in the contact geometry setting. In particular, we utilize
the normalization of the Conley--Zehnder index in which the index of a
non-degenerate maximum with small eigenvalues in dimension $2n$ is
equal to $n$. All homology groups are taken with rational coefficients
unless stated otherwise. The global contact homology is always
understood to be the linearized contact homology. The Floer homology,
local or global, is graded by the Conley--Zehnder index, while the
contact homology in dimension $2n-1$ is, as is customary, graded by
the Conley--Zehnder index of the return map plus $n-3$; the Morse
homology is graded by the Morse index.  We use the notation $\MUCZ(x)$
for the Conley--Zehnder index of a periodic orbit $x$ of a Hamiltonian
diffeomorphism and $|x|$ for the degree of a closed Reeb orbit
$x$. The mean index of $x$ is denoted by $\Delta_H(x)$, where $H$ is a
Hamiltonian, or by $\Delta(x)$. (See, e.g., \cite{Lo,SZ}, for
miscellaneous properties of the Conley--Zehnder and mean indexes used
throughout the paper.)  

\medskip
\noindent\textbf{Acknowledgements.} The authors are grateful to Peter
Albers, Fr\'ed\'eric Bourgeois, Ba\c sak G\"urel, and
Alex Oancea for useful discussions.

\section{Local analysis}
\label{sec:local}

\subsection{Local contact and Floer homology}
\label{sec:comparison}
Consider a closed orbit $x$ of the Reeb flow for a contact form
$\alpha$ on a manifold $M^{2n-1}$.  We do not require $x$ to be
simple, i.e., $x$ can be a multiply-covered, iterated orbit. However,
we do require $x$ to be isolated in the loop space. In this setting,
the local contact homology $\HC_*(x)$ of $\alpha$ at $x$ is
defined. This is the homology of a complex $\CC_*(\alpha',
\ka [S^1])$ generated by some of the closed Reeb orbits which $x$
splits into under a non-degenerate perturbation $\alpha'$ of $\alpha$.

To be more precise, assume that $x=y^\ka$, where $y$ is a simple
closed Reeb orbit of $\alpha$. Fix a small tubular isolating
neighborhood $U=B\times S^1$ of $x$ (or, from a geometrical
perspective, of $y$) and let $\alpha'$ be a $C^\infty$-small
non-degenerate perturbation of $\alpha$. Consider closed Reeb orbits
of $\alpha'$ in the homotopy class $\ka S^1$, contained in $U$. Note
that some of these orbits can be iterated. We call an orbit bad if it
is an even iteration of an orbit with an odd number of Floquet
multipliers in $(-1,\, 0)$.  Otherwise, an orbit is said to be
good. (Note that all orbits in question are non-contractible in $U$.)
The complex $\CC_*(\alpha',\ka[S^1])$ is generated by the good closed
orbits of~$\alpha'$ in the free homotopy class $\ka[S^1]$.  (It is not
hard to show that the solutions of the Cauchy--Riemann equation in the
symplectization of $U$, asymptotic to such orbits, stay away from the
boundary region in $U$; hence, compactness. See \cite{HM} for more
details.) When $\kappa>1$, making this construction rigorous
encounters fundamental difficulties inherent in the definition of the
contact homology: the symplectization of $U$ need not admit an almost
complex structure $J$ meeting the regularity requirements, see
\cite{Bou,SFT,HM}. We proceed assuming that such an almost complex
structure $J$ exists.

On the other hand, consider the Poincar\'e return map $\varphi_x$ of
$x$. This is the germ of a Hamiltonian diffeomorphism with isolated
fixed point, which we still denote by $x$. Then we also have the local
Floer homology of $\varphi_x$ at $x$ defined; see
\cite{Gi:CC,GG:gap,McL}. We denote this homology by $\HF_*(x)$ or
$\HF_*(\varphi_x)$, depending on the context.  When $x$ is simple, the
groups $\HC_*(x)$ and $\HF_*(x)$ are isomorphic up to a shift of
degree; see \cite[Proposition 5.1]{HM}. (This can also be established
by repeating word-for-word the proof of \cite[Proposition 4.30]{EKP}
-- the fact that the Floer homology of a stable Hamiltonian structure
is independent of a framing and of an adjusted almost complex
structure.)  Our next goal is to show that this is still true in
general, once the ordinary local Floer homology group $\HF_*(x)$ is
replaced by its equivariant counterpart.

To describe the setting more accurately, let us first focus on the
case of global Floer homology. Thus let $W$ be a closed symplectic
manifold, which for our purposes can be required to be symplectically
aspherical. Consider a one-periodic in time Hamiltonian $H\colon
W\times S^1 \to \R$, where $S^1=\R/\Z$.  For any positive integer
$\ka$, the Hamiltonian $H$ can also be treated as $\ka$-periodic. In
this case, we will use the notation $H^{\nat \ka}\colon W\times
S^1_\ka\to \R$, where $S^1_\ka=\R/\ka\Z$, and, abusing terminology,
call $H^{\nat \ka}$ the $\ka$th iteration of $H$. The critical points
of the action functional for $H^{\nat\ka}$ are precisely the
contractible $\ka$-periodic orbits of $H$ and $\Z_\ka$ acts on the set
of $\ka$-periodic orbits by time-shift. Let us first assume that $H$
is strongly non-degenerate, i.e., all Hamiltonians $H^{\nat\ka}$ are
non-degenerate.  The definition of the $\Z_\ka$-equivariant Floer
homology $\HF_*^{\Z_\ka}(H^{\nat\ka})$ is identical, essentially
word-for-word with obvious modifications, to the construction of the
$S^1$-equivariant Floer homology for autonomous Hamiltonians; see
\cite[Section 5]{Vi} and \cite{BO:gysin,BO:transv}. (In Section
\ref{sec:Morse} we illustrate the construction by the ``elementary''
case of the equivariant Morse homology.) The resulting homology has
the expected properties similar to the standard Floer homology. In
particular, we also have the filtered equivariant Floer homology,
continuation maps, and, by continuity, the construction extends to
all, not necessarily non-degenerate, Hamiltonians. The total
$\Z_\ka$-equivariant Floer homology is independent of $H$ and does not
appear to carry much new information when $W$ is closed. Over $\Q$, we
have $\HF_*^{\Z_\ka}(H^{\nat\ka})\cong \H_*(W)$ since the group
homology of $\Z_\ka$ vanishes, i.e., $\H_{*>0}(\Z_\ka)=0$. Note
however that the filtered $\Z_\ka$-equivariant Floer homology can
differ from its non-equivariant counterpart.

Coming back to our discussion of the local Floer and contact homology,
observe that the construction of $\Z_\ka$-equivariant Floer homology
localizes in a straightforward way.  Namely, assume that $x=y^\ka$,
i.e., $x$ is the $\ka$th iteration of a simple orbit $y$. Then
$\varphi_x=\varphi_y^\ka$, and the $\Z_\ka$-equivariant local Floer
homology $\HF_*^{\Z_\ka}(\varphi_x)=:\HF_*^{\Z_\ka}(x)$ is
defined. In general, $\HF_*^{\Z_\ka}(x)\neq \HF_*(x)$ even
over $\Q$.

Finally, to explicitly relate the local contact and Floer homology, we
need to fix the gradings of these groups. To this end, we pick a
symplectic trivialization of $\ker\alpha\!\mid_y$. Such a trivialization
determines the grading of the local contact homology and also enables
us to view $\varphi_x$ and $\varphi_y$ as elements of the universal
covering of the group of the germs of Hamiltonian diffeomorphisms,
which in turn determines the grading of the local Floer homology. We
have $|x|=\MUCZ(x)+(n-3)$.

\begin{Theorem}
\label{thm:ch-vs-fh}
Let $x=y^\ka$ be an isolated closed Reeb orbit. Then we have 
$\HC_*(x)=\HF_{*+n-3}^{\Z_\ka}(x)$.
\end{Theorem}

As readily follows from the construction of the equivariant Floer
homology (see \cite[Section 5]{Vi} and \cite[Section 5]{BO:transv} and
also Section \ref{sec:Morse}), there exists a spectral sequence with
$E^2_{*,*}=\HF_*(x)\otimes \H_*(\Z_\ka)=\HF_*(x)$ converging to
$\HF_*^{\Z_\ka}(x)$. As a consequence, we obtain

\begin{Corollary}[\cite{HM}]
\label{cor:rnak}
The total dimension, over $\Q$, of the local contact homology does not exceed the
total dimension of the local Floer homology: $\dim_\Q \HC_*(x)\leq
\dim_\Q \HF_*(x)$. Furthermore, $\HC_*(x)=\HF_{*+n-3}(x)$ when the orbit
$x$ is simple, i.e., $\ka=1$.
\end{Corollary}

\begin{Remark}
\label{rmk:cor-vs-thm}
It is worth keeping in mind that even though we establish here
Corollary \ref{cor:rnak} as a consequence of Theorem
\ref{thm:ch-vs-fh}, the corollary can be proved directly; see
\cite{HM} for a very detailed argument.
\end{Remark}

\begin{proof}[Outline of the proof of Theorem \ref{thm:ch-vs-fh}]
  As above, let $U=B\times S^1$ be a small isolating tubular
  neighborhood of $y$ and let $\alpha'$ be a $C^\infty$-small
  non-degenerate perturbation of $\alpha$. Recall that $\HC_*(x)$
  is, by definition, the homology of the complex $\CC_*(\alpha',\ka
  [S^1])$ generated by the good orbits of $\alpha'$ in $U$ in the
  homotopy class $\kappa[S^1]$ and that we assume that the symplectization of $U$
  admits an almost complex structure $J$ such that the necessary
  regularity requirements are satisfied. 
  
  Consider the $k$-fold covering $\tilde{U}=B\times S^1_\ka$ of $U$. We
  denote by $\tilde{\alpha}'$ and $\tilde{J}$ the lifts of $\alpha'$
  and $J$ to $\tilde{U}$. The group of deck transformations $\Z_\ka$
  acts on the contact complex $\CC_*(\tilde{\alpha}', [S^1_\ka])$ for
  the homotopy class $[S^1_\ka]$. Namely, the
  shift $g\colon t\mapsto (t+g)$ on $S^1_\ka$ acts by sending a
  periodic orbit $z$ of $\tilde{\alpha}'$ to the orbit $\pm g(z)$; see
  \cite[Section 6.3]{HM}.  The regularity of $J$ guarantees that this
  action commutes with the differential and, again essentially by
  definition, the contact complex $\CC_*(\alpha', \ka[S^1])$ agrees
  with the invariant part $\CC_*(\tilde{\alpha}', [S^1_\ka])^{\Z_\ka}$
  of the complex for $\tilde{\alpha}'$. Hence
\begin{equation}
\label{eq:c-c}
\HC_*(x):=\HC_*(\alpha',\ka[S^1])=\HC_*(\tilde{\alpha}',[S^1_\ka])^{\Z_\ka}
=:\HC_*(\tilde{x})^{\Z_\ka},
\end{equation}
where $\tilde{x}$ is the inverse image of $x$ in $\tilde{U}$; cf.\
\cite[Section 6]{HM}. 

\begin{Example}[Good vs.\ bad]
\label{ex:good-vs-bad}
  Assume that $x=y^2$ and $x$ is non-degenerate. Let $g$ be the
  non-trivial element in $\Z_2$. Then $g\colon
  \tilde{x}\mapsto\tilde{x}$ when $x$ is good and $g\colon
  \tilde{x}\mapsto -\tilde{x}$ when $x$ is bad. Thus, in the former
  case, we have $\HC_*(x)=\HC_*(\tilde{x})=\Q$ (supported in degree
  $\MUCZ(x)+n-3$). In the latter case, $\HC_*(x)=0$, but
  $\HC_*(\tilde{x})=\Q$.
\end{Example}

The rest of the proof relies heavily, at least on the conceptual
level, on the constructions of the invariant and equivariant Morse
homology outlined in Section \ref{sec:Morse} and on the proof of
Proposition \ref{prop:Morse}.  Namely, similarly to the definition of
the invariant Morse homology, one can define the $\Z_\ka$-invariant
(local) Floer homology for an iterated Hamiltonian or an iterated
orbit, provided, of course, that there exists a one-periodic almost
complex structure satisfying the regularity requirements.  Hence we
have the $\Z_\ka$-invariant local Floer homology
$\HF_*(\varphi_{\tilde{x}})^{\Z_\ka}$ of the Poincar\'e return map of
$\tilde{x}$. Clearly, $\varphi_{\tilde{x}}=\varphi_x=\varphi_y^\ka$
and $\HF_*(\varphi_{\tilde{x}})^{\Z_\ka}=\HF_*(\varphi_x)^{\Z_\ka}$.
A homotopy between a contact framing in $\tilde{U}$ and a Hamiltonian
framing (see \cite[Section 5]{HM} and \cite[Proposition 4.30]{EKP})
induces an isomorphism
$\HC_*(\tilde{x})=\HF_*(\varphi_{\tilde{x}})$. When the homotopy is
regular and $\Z_\ka$-invariant, the continuation map commutes with the
$\Z_\ka$-action, and we obtain an isomorphism
 \begin{equation}
 \label{eq:c-f}
\HC_*(\tilde{x})^{\Z_\ka}=\HF_{*+n-3}(\varphi_{\tilde{x}})^{\Z_\ka}=:
\HF_{*+n-3}(\varphi_x)^{\Z_\ka}.
 \end{equation}

Furthermore, Proposition \ref{prop:Morse} establishes an isomorphism between
invariant and equivariant Morse homology. With suitable modifications,
the proposition and its proof carry over to the realm of
Floer homology, again provided that the required regularity conditions
are met. As a consequence, we have 
\begin{equation}
\label{eq:f-f}
\HF_*(\varphi_x)^{\Z_\ka}=
\HF_*^{\Z_\ka}(\varphi_x).
\end{equation}

Combining the isomorphisms \eqref{eq:c-c}, \eqref{eq:c-f}, and
\eqref{eq:f-f}, we conclude that
$\HC_*(x)=\HF_{*+n-3}^{\Z_\ka}(x)$.
\end{proof}

\begin{Remark}
\label{rmk:multi-valued}
Eliminating the assumption that the symplectization of $U$ admits a
regular almost complex structure $J$, crucial for the proof of Theorem
\ref{thm:ch-vs-fh}, requires the machinery of multi-valued
perturbations in the context of contact homology.  Developing such a
machinery is currently work in progress by Hofer, Wysocki and Zehnder;
see \cite{HWZ3,HWZ4} and references therein, and also Remark
\ref{rmk:multi-valued2}.
\end{Remark}

\subsection{Symplectically degenerate maxima}
\label{sec:sdm}

In this section, we recall several relevant results concerning symplectically
degenerate maxima, following mainly \cite[Section 5.1]{GG:gap}. Let
$\varphi=\varphi_H$ be an element of the universal cover of the group
of germs of Hamiltonian diffeomorphisms of $\R^{2d}$ at a point
$x$. Throughout this section, it will be convenient to assume that
$x$ is an isolated fixed point for all iterations
$\varphi^\ka$. (Treating $x$ as a one-periodic orbit, we will
sometimes write $x^\ka$ to indicate the order of iteration.) Recall that
$\Delta_H(x)$ is the mean index of $x$. 

\begin{Definition}
\labell{def:sdm}
The point $x$ is said to be a \emph{symplectically degenerate
maximum} (or an SDM) if $\Delta(x)=0$ and $\HF_d(\varphi_H)\neq 0$.
\end{Definition}

Note that requiring $\varphi$ to be an element of the universal cover
rather than just a Hamiltonian diffeomorphism makes
the mean index of $x$ and the grading of the local local Floer
homology well defined. (Otherwise, $\Delta(x)$ and the grading would
be defined only modulo $2\Z$.)  A germ of a Hamiltonian diffeomorphism $\varphi$
is called an SDM if it admits an SDM lift to the universal cover.
This is equivalent to that $\Delta_H(x)\in 2\Z$ and
$\HF_{d+\Delta_H(x)}(\varphi_H)\neq 0$ for any (or,
equivalently, some) Hamiltonian $H$ generating~$\varphi$.

\begin{Example}
\labell{exam:sdm} 
Let $H$ be an autonomous Hamiltonian attaining a strict local maximum
at $x$.  Assume in addition that $\d^2 H_x=0$ or, more generally, that
all eigenvalues of $\d^2 H_x$ are zero.  Then, as is easy to see, $x$
is an SDM, cf.\ Proposition \ref{prop:sdm2}. (Here, as is customary in
Hamiltonian dynamics, the eigenvalues of a quadratic form on a
symplectic vector space are, by definition, the eigenvalues of the
linear symplectic vector field it generates.  Thus, for instance, the
quadratic form $Q(p,q)=q^2$ on the standard $(\R^2,dp\wedge dq)$ has
zero eigenvalues.)
\end{Example}

Recall that $\varphi$ (or $x$) is said to be totally degenerate if all
eigenvalues of $D\varphi$ are equal to $1$. An iteration $\ka$ is
called admissible if the generalized eigenvalue $1$ has the same
multiplicity for $D\varphi$ and $D\varphi^\ka$. For instance, all
$\ka$ are admissible when $\varphi$ is totally degenerate or
when none of the eigenvalues of $D\varphi$ is a root of
unity. Although this fact is not explicitly used in what follows, it is
useful to keep in mind that an admissible iteration of an isolated
orbit is automatically isolated; see \cite[Proposition 1.1]{GG:gap}.

\begin{Proposition}[\cite{GG:gap}]
\label{prop:sdm1}

The following conditions are equivalent:

\begin{itemize}

\item[(a)] the point $x$ is an SDM;

\item[(b)] $\HF_d(\varphi_H^{\ka_i})\neq 0$ for some sequence of
admissible iterations $\ka_i\to\infty$;

\item[(c)] the point $x$ is totally degenerate,
$\HF_d(\varphi_H)\neq 0$ and $\HF_d(\varphi_H^{\ka})\neq 0$ for at
least one admissible $\ka\geq d+1$;

\item[(d)] $\HF_d(\varphi_H^{\ka})=\Q$ and $\HF_{*\neq
  d}(\varphi_H^\ka)=0$ for all $\ka$.

\end{itemize}

\end{Proposition}

Equivalence of (a) and (b) and (c) is proved in \cite{GG:gap}. Clearly
(d) implies (b). The fact that (d) also follows from (b) can be easily
extracted from the proof of \cite[Theorem 1.1]{GG:gap} combined with
\cite[Section 3.1]{Gi:CC}).

We observe that, as a consequence of the proposition and of the
persistence of local Floer homology (see \cite[Theorem 1.1]{GG:gap}),
the point $x^\ka$ is an SDM for any admissible iteration $\ka$ if and
only if $x$ is an SDM.  The next proposition shows that the behavior
of $\varphi_H$ near an SDM is similar to that described in Example
\ref{exam:sdm}. The essence of this result is that $x$ is an SDM if
and only if $\varphi$ can be generated by a Hamiltonian $K$ with local
maximum at $x$ and arbitrarily small Hessian.

\begin{Proposition}[\cite{Gi:CC,Hi}]
\label{prop:sdm2}

The point $x$ is an SDM if and only if for every $\eta>0$ there exists
a Hamiltonian $K$ near $x$ such that $\varphi_K=\varphi_H$ in the
universal covering of the group of germs of Hamiltonian
diffeomorphisms at $x$ and
\begin{itemize}
\item[(i)] $x$ is a strict local maximum of $K_t$ for all $t\in S^1$,
\item[(ii)] $\d^2 (K_t)_x$ has zero eigenvalues and $\|\d^2
  (K_t)_x\|_{\Xi}<\eta$ for all $t\in S^1$ and some symplectic basis
  $\Xi$ in $T_p M$.
\end{itemize}

\end{Proposition}

Here $\|\d^2 (K_t)_x\|_{\Xi}$ stands for the norm of $\d^2
(K_t)_x$ with respect to the Euclidean inner product on $T_x \R^{2n}$
for which $\Xi$ is an orthonormal basis; see \cite[Section
2.1.3]{Gi:CC}. 

\begin{Remark}
\label{rmk:min}
Replacing $d$ by $-d$ in the above discussion, we arrive at the notion
of a symplectically degenerate minimum (SDMin); see \cite{He:cotangent}. It has
properties similar to those of an SDM but with a maximum replaced
by a minimum. Note that $x$ is an SDM for $\varphi$ if and only if $x$
is an SDMin for $\varphi^{-1}$.
\end{Remark}

An isolated closed Reeb orbit $x$ of $\alpha$ on $M^{2n-1}$ is said to
be an SDM if it is an SDM for the Poincar\'e return map
$\varphi$. Note that this notion depends on the choice of a
trivialization of $\ker\alpha\!\mid_x$. For this trivialization gives
rise to a lift of $\varphi$ to the universal cover of the group of
germs of Hamiltonian diffeomorphism.  Applying Proposition
\ref{prop:sdm1}, and keeping in mind the difference in gradings and
that $d=n-1$, we obtain

\begin{Proposition}
\label{prop:sdm-Reeb}

The following conditions are equivalent:

\begin{itemize}

\item[(a)] the orbit $x$ is an SDM with respect to a trivialization
  such that $\Delta(x)=0$;

\item[(b)] $\HC_{2n-4}(x^{\ka_i})\neq 0$ for some sequence of
admissible iterations $\ka_i\to\infty$;

\item[(c)] the orbit $x$ is totally degenerate,
$\HC_{2n-4}(x)\neq 0$ and $\HC_{2n-4}(x^{\ka})\neq 0$ for at
least one admissible iteration $\ka\geq n$;

\item[(d)] $\HC_{2n-4}(x^{\ka})=\Q$ and $\HC_{*\neq
  2n-4}(x^\ka)=0$ for all $\ka$.

\end{itemize}

\end{Proposition}

\subsection{Local Floer and contact homology in low dimensions}
\label{sec:low-dim}
In this section, we state some simple properties of local Floer
homology in dimension two and local contact homology in dimension
three to be used in what follows. These properties are easy
consequences of the results from \cite{GG:gap} and Theorem
\ref{thm:ch-vs-fh}.

Let $x$ be a closed Reeb orbit in
dimension three (with a fixed trivialization of $\ker\alpha$ along
$x$) and let $\varphi$ be an element of the universal cover of the
group of germs of Hamiltonian diffeomorphisms of $\R^{2}$ at a point
or more generally a periodic orbit,
which we also denote by $x$. (For instance, we can have
$\varphi=\varphi_x$ in the setting of Section \ref{sec:comparison}.)
We always assume that $x^\ka$ is isolated for all
$\ka$. As in Section \ref{sec:comparison}, we will use the notation
$\HF_*(x)$ or $\HF_*(\varphi)$ and $\HC_*(x)$ for the local Floer and
local contact homology of $x$.

We start by recalling the following particular case of \cite[Theorem
1.1]{GG:gap} mentioned in Section \ref{sec:sdm} and repeatedly used in
this paper:

\begin{Proposition}[Persistence of Local Floer Homology, \cite{GG:gap}]
\label{prop:persistence}
Let, as above, $x$ be a periodic orbit (simple or iterated)
of $\varphi$ in dimension two. Assume that $x$ is degenerate. Then
$$
\HF_*(\varphi^\ka)=\HF_{*+(\ka-1)\Delta(x)}(\varphi).
$$
\end{Proposition}

Furthermore, local Floer and local contact homology in low dimensions have
the following features:

\begin{itemize}

\item[(LF1)] The local homology groups $\HF_*(\varphi)$ are
  concentrated in at most one degree $k\in [\Delta(x)-1,\Delta(x)+1]$,
  i.e., $\HF_k(\varphi)\neq 0$ for at most one $k$ and
  $|k-\Delta(x)|\leq 1$. When $x$ is non-degenerate $k=\MUCZ(x)$. In
  general, we will call $k$ the generalized Conley--Zehnder index of
  $x$ and keep the notation $k=\MUCZ(x)$. Likewise, $\HC_*(x)$ is
  concentrated in at most one degree $k\in
  [\Delta(x)-2,\Delta(x)]$. When $x$ is non-degenerate and the local
  homology does not vanish, $k=|x|$. In general, we keep the notation
  $|x|$ for this degree, when $\HC_*(x^\ka)\neq 0$.

\end{itemize}
When $x$ is non-degenerate this is obvious. In the degenerate case,
the result in the case of Floer homology can be easily extracted from
the proof of \cite[Theorem 1.1]{GG:gap} and the fact that the local
Morse homology of a function in two variables is concentrated in at
most one degree. The assertion for contact homology now readily
follows from Theorem \ref{thm:ch-vs-fh} or even from Corollary
\ref{cor:rnak}, cf.\ Remark \ref{rmk:cor-vs-thm}.

\begin{itemize}

\item[(LF2)] Assume that $x$ is degenerate and $\HF_*(x)\neq 0$ (or
  $\HC_*(x)\neq 0$ in the contact case). Then the mean index
  $\Delta(x)$ is an even integer and, in the notation of (LF1),
  $\MUCZ(x^\ka)=1+\ka\Delta(x)+\eps$ (or 
$|x^\ka|=\ka\Delta(x)+\eps$), where $\eps$ is either $\pm 1$ or
  0, and is independent of $\ka$. 

\end{itemize}
Here only the assertion that $\eps$ is independent of $\ka$ requires
a proof. In the case of Floer homology, this fact can be
readily extracted from the proof of the persistence of local Floer
homology; see \cite[Theorem 1.1]{GG:gap}. The assertion in the contact
case, follows now from Theorem \ref{thm:ch-vs-fh} or Corollary
\ref{cor:rnak}. Note that at this point we can only claim
that $\HC_*(x^\ka)$ must vanish for $*\neq |x^\ka|$, but not that
$\HC_{|x^\ka|}(x^\ka)\neq 0$; see however \eqref{eq:c-persistence}.

Let us now take a closer look at the situation when the local Floer or
contact homology of $x$ or its iterations vanish. Assume first that
$x$ is degenerate. Then, as is easy to see, we can have $\HF_*(x)=0$
in all degrees. (In all examples known to us, such homologically
trivial orbits are simple.) By Corollary \ref{cor:rnak}, 
$\HC_*(x)=0$, when $x$ is a closed Reeb orbit with
$\HF_*(x)=0$. Moreover, $\HF_*(x^\ka)=0$ and $\HC_*(x^\ka)=0$ for all
$\ka$. Hypothetically, it is also possible that $\HC_*(x)=0$
while $\HF_*(x)\neq 0$, but we do not have any examples of
degenerate orbits for which this is the case.  When $x$ is non-degenerate,
clearly $\HF_{\MUCZ(x)}(x)\neq 0$. However, 
$\HC_*(x)$ can vanish even when $x$ is non-degenerate. This is the case if
and only if $x$ is an even iteration of a hyperbolic orbit with
negative Floquet multipliers; see Example \ref{ex:good-vs-bad}.

Note that when the local Floer or local contact homology is trivial, the
definition of $\MUCZ(x)$ or of $|x|$ from (LF1) does not apply. When
$\HF_*(x)\neq 0$, e.g., in the non-degenerate case, we still set
$|x|=\MUCZ(x)-1$. Alternatively, in Section \ref{sec:S^3}, we
use, as a matter of convenience, a suitable \emph{ad hoc} definition of
$|x|$.

The next fact, stated for the sake of brevity only in the contact
case, is an immediate consequence of Proposition \ref{prop:sdm1} (and
Theorem \ref{thm:ch-vs-fh} applied to an SDM).

\begin{itemize}

\item[(LF3)] A simple orbit $x$ is an SDM (with respect to some
  trivialization of $\ker\alpha\!\mid_x$) if and only if one of following
  conditions is satisfied:

\begin{enumerate}

\item[(i)] $\HC_{2m\ka_i}(x^{\ka_i})\neq 0$ (or
  equivalently $\HC_{2m\ka_i}(x^{\ka_i})=\Q$) for some $m\in\Z$ and some sequence of
  admissible iterations $\ka_i\to\infty$;

\item[(ii)] $x$ is degenerate and $\HC_{2m}(x)\neq 0$ (or
  equivalently $\HC_{2m}(x)=\Q$), where $\Delta(x)=2m$ and $m\in\Z$.

\end{enumerate}

\end{itemize}

\section{Proof of Theorem \ref{thm:S^3}}
\label{sec:S^3}

Our goal in this section is to derive Theorem \ref{thm:S^3} from
Theorem \ref{thm:sdm}.  
Thus let $\alpha$ be a contact form supporting the standard 
contact structure on $S^3$. The key property of $\alpha$  we need is that $\HC_*(\alpha)$ is concentrated in even
degrees $2,4,\ldots$ and that $\HC_*(\alpha)=\Q$ in these degrees; see, e.g.,
\cite{Bou} and references therein. (Here the indices
are evaluated with respect to a global symplectic trivialization of
$\ker\alpha$ on $S^3$; such a trivialization is unique up to
homotopy.)

Arguing by contradiction, let us assume that $\alpha$ has only one
simple closed Reeb orbit $x$. We will show that $x$ is then an SDM, and
hence, by Theorem \ref{thm:sdm}, the Reeb flow of $\alpha$ has
infinitely many periodic orbits.

First, observe that, by (LF1), every iteration $x^\ka$ can contribute
to the homology in at most one degree
$|x^\ka|=\MUCZ(x^\ka)-1$, even when $x^\ka$ is degenerate. In particular,
for some values of $\ka$ we must have $|x^\ka|=2,4,\ldots$ or
equivalently $\MUCZ(x^\ka)=3,5,\ldots$.

As an immediate consequence, we see that $x$ cannot be
hyperbolic. Indeed, if $x$ is hyperbolic and its Floquet multipliers
are positive, $\MUCZ(x^\ka)$ is even for all $\ka$. If the Floquet
multipliers are negative, we have $\MUCZ(x^\ka)=\ka(2\ell+1)$, for
some integer $\ell$, and $\dim\HF_{\ka(2\ell+1)}(x^\ka)=1$, which is
also impossible. (In fact, as is easy to see, we would have $\dim
\HC_{\kappa(2l+1)}(x^\kappa)=1$ if $\kappa$ is odd and
$\HC_{\kappa(2l+1)}(x^\kappa)=0$ if $\kappa$ is even.)

Thus $x$ is elliptic, i.e., its Floquer multipliers have absolute
value one. Let $\Delta=2m+2\lambda$, where $m\in \Z$ and $\lambda\in
[0,\,1)$, be the mean index of $x$. Then $\Delta(x^\ka)=\ka\Delta$
and $0<\Delta<3$.

Next let us show that $x$ cannot be strongly non-degenerate. (An orbit
is strongly non-degenerate when all its iterations are
non-degenerate.)  First, observe $x^\ka$ is strongly non-degenerate if
and only if $\lambda\not\in \Q$. Furthermore, in this case
$\MUCZ(x^\ka)=1+2m\ka+2\lfloor \ka\lambda \rfloor$ (see, e.g.,
\cite{Lo}), and hence
$$
|x^\ka|=2m\ka+2\lfloor
\ka\lambda \rfloor.
$$
As a consequence, $|x^\ka|$ assumes only even values starting with
$|x|=2$. (Hence $m=1$.) However, some arbitrarily large even degrees
are skipped when $\lfloor \ka\lambda\rfloor$ jumps from one integer
value to the next. This is impossible. Alternatively, the fact that
$x$ cannot be strongly non-degenerate, including the hyperbolic case,
also follows immediately from the results of \cite {HWZ:03}. 

Next, recall that in general, when $\HC_*(x^\ka)\neq 0$, we denote by
$|x^\ka|$ the degree in which this space does not vanish. (By (LF1),
this degree is unique.) We have
\begin{equation}
\label{eq:deg}
|x^\ka|=2m\ka+2\lfloor
\ka\lambda \rfloor+\eps_\ka,
\end{equation}
where $\eps_\ka$ is either $-2$ or $-1$ or $0$. (When $x^\ka$ is
non-degenerate, $\eps_\ka=0$.) For all $\ka$ such that
$x^\ka$ is degenerate, $\eps_\ka$ takes the same
value due to (LF2). 

When $x^\ka$ is degenerate, we can have $\HC_*(x^\ka)=0$ in all
degrees; see Section \ref{sec:low-dim}. In this case, we use
\eqref{eq:deg} with, say, $\eps_\ka=0$ to define $|x^\ka|$. Note that
with this convention $\eps_\ka$ is still independent of $\ka$ when
$x^\ka$ is degenerate.

Observe now that necessarily $m=0$ or $m=1$.

Assume first that $m=1$. Then either $\lambda=0$ or $\lambda>0$. In
the former case, by (LF3), $x$ is an SDM (with respect to a twisted
trivialization along $x$) as required. In the latter case, we can
assume that $\lambda$ is rational. Then
$\lambda=1/\ell$, for some integer $\ell >1$. (Otherwise, $|x^\ka|$ would
skip an even value when $\lfloor \ka\lambda\rfloor$ jumps from $0$ to
$1$.) Thus $|x^\ka|$ is the sequence
$$
2,  \ldots, 2(\ell-1), \, 2\ell+\eps_\ell, \,
2(\ell+1)+2,\ldots, 2(2\ell-1)+2,\, 4\ell+2+\eps_{2\ell}, \,2(2\ell+1)+4,\ldots,
$$ 
which is clearly impossible.

It remains to examine the case $m=0$. Then $\Delta=\lambda\in (0,\,1)$
and again $\lambda\in \Q$. As in the
case $m=1$, it is easy to see that $\lambda=1/\ell$ for some integer
$\ell>1$. As a consequence, $x^\ka$ degenerates when $\ka$ is
divisible by $\ell$, and 
\begin{equation}
\label{eq:degs}
|x^\ka|=\underbrace{0,\ldots,0}_{\ell-1},\,1,\,
\underbrace{2,\,\ldots,2}_{\ell-1},\,3,\ldots .
\end{equation}

There are two ways to rule out this possibility. (We are not aware of
any example of the germ of a Hamiltonian diffeomorphism with this
behavior, but we do not have a ``local'' proof that this is impossible.)  

The first one is to invoke some of the results of Hofer--Wysocki--Zehnder;
see \cite{HWZ1,HWZ2}. Namely, by \cite[Theorems 1.7 and 1.8]{HWZ2},
there exists a non-constant finite energy plane in the symplectization
of $S^3$, injective in $S^3$ and asymptotic to a simple closed Reeb orbit
$y$. By \cite[Theorem 4.3]{HWZ1}, $\MUCZ(y)\geq
2$. However, in our setting, the only simple orbit is $x$, which is
non-degenerate, and $\MUCZ(x)=1$.

Alternatively, using the full strength of Theorem \ref{thm:ch-vs-fh} and
more, one can argue as follows. First, one shows that in dimension
three
\begin{equation}
\label{eq:c-persistence}
\dim \HC_*(z^\ka)=\dim\HC_*(z)
\end{equation}
up to a shift of degree, when $z$ is degenerate. Taking into account
Theorem \ref{thm:ch-vs-fh}, we can express \eqref{eq:c-persistence} as
$\dim\HF_*^{\Z_{\ka\nu}} (z^\ka)= \dim\HF_*^{\Z_{\nu}} (z)$, again up
to a shift, where $z=y^\nu$ and $y$ is simple. In fact,
\eqref{eq:c-persistence} is a contact (or equivariant) analogue of
Proposition \ref{prop:persistence}.  Using \eqref{eq:degs} and
\eqref{eq:c-persistence} for $z=x^\ell$ and setting $r=\dim \HC_*(z)$,
we see that the spaces $\HC_*(x^\ka)$ fit together to form a complex
$$
0\leftarrow\Q^{\ell-1}\leftarrow \Q^r\leftarrow\Q^{\ell-1}\leftarrow
\Q^r\leftarrow\cdots
$$
with the first non-zero term $\Q^{\ell-1}$ occurring in degree
zero. The homology of this complex is $\HC_*(\alpha)$, i.e., it is
$\Q$ in every even degree starting with two and zero otherwise. We
claim that this is impossible. Indeed, let $p_j\geq 0$ be the rank of the
differential $\Q^r\leftarrow\Q^{\ell-1}$ from degree $2j$ to degree
$2j-1$. Thus, for instance, $p_1=r-(\ell-1)$. Now, a simple inductive
argument shows that $p_j=j p_1+ (j-1)\to\infty$, which contradicts 
the obvious restriction $p_j\leq r$.

\section{Proof of Theorem \ref{thm:sdm}}
\label{sec:prf-sdm}
The argument is a rather faithful adaptation of the proof of the
Conley conjecture given in \cite{GG:gaps} to the contact case. Hence
we only briefly outline the proof skipping numerous technical details.

\subsection{Filtered contact homology in the presence of an SDM}
\label{sec:ch-sdm}
As in the proof of the degenerate case of the Hamiltonian Conley
conjecture (see, e.g., \cite{Gi:CC,GG:gaps,Hi,He:CC}), Theorem
\ref{thm:sdm} is a consequence of the fact that the presence of an SDM
forces the filtered contact homology to be non-zero for certain
arbitrarily small action intervals located right above the action of
the SDM iterations.

To be more precise, let $(M^{2n-1},\ker\alpha)$ be a contact manifold
admitting a strong symplectically aspherical filling $(W,\omega)$. As
in Theorem \ref{thm:sdm}, assume that the Reeb flow of $\alpha$ has a
simple closed Reeb orbit $x$ which is a symplectically degenerate
maximum with respect to some symplectic trivialization of
$\ker\alpha\!\mid_x$.  Denote by $\HC_*(M,\alpha)$ the (filtered)
contact homology of $(M,\alpha)$, linearized by means of $W$, for the
collection of free homotopy classes $\{ [x^\ka]\mid \ka\in \N\}$. When
$x$ is contractible in $W$, the condition that
$c_1(TW)\!\mid_{\pi_2(W)}=0$ ensures that $\HC_*(M,\alpha)$ has a
well-defined grading. If $x$ is not contractible, we 
require, in addition, $(W,\omega)$ to be atoroidal, and fix a symplectic
trivialization of $TW\!\mid_{x^\ka}$ to have the grading well defined.
(In fact, it would be sufficient to fix, a trivialization of the
complex line bundle $(\bigwedge^n_\C
TW\!\mid_{x^\ka})^{\otimes 2}$; see, e.g., \cite{Bou,SFT,Es} for more
details.) To be more specific, in this case, we pick a trivialization of $TW\!\mid_x$
and equip $TW\!\mid_{x^\ka}$, for $\ka\geq 2$, with ``iterated
trivializations''. (When $x$ is contractible, it is equipped with a
trivialization which extends to a disk bounded by $x$.) From now on,
all degrees and (mean) indices are taken with respect to such fixed
reference trivializations unless explicitly stated otherwise.

Assume that all iterations $x^\ka$ are isolated. Let $2m$, where
$m\in\Z$, be the Maslov index of the trivialization with respect to
which $x$ is an SDM. Thus, by Proposition \ref{prop:sdm-Reeb},
$\HC_*(x^\ka)$ is concentrated and equal to $\Q$ in degree
$d_\ka=2m\ka+2n-4$. Denote by $c$ the action of $x$.  Then the action
of $x^\ka$ is $\ka c$.

\begin{Theorem}
\label{thm:ch-sdm}
For any sufficiently small $\eps>0$, 
\begin{equation}
\label{eq:hom}
\HC_{d_\ka+1}^{(\ka c,\,\ka c+\eps)}(M,\alpha)\neq 0
\end{equation}
for all large iterations $\ka>\ka_0(\eps)$.
\end{Theorem}

\begin{Remark}
\label{rmk:interval}
Strictly speaking here, as in similar results in the Hamiltonian
setting (see, e.g., \cite[Proposition 4.7]{Gi:CC}, \cite[Theorem
1.7]{GG:gaps} and \cite[Theorem 1.5]{He:CC}), one has to replace the
action interval $(\ka c,\,\ka c+\eps)$ by $(\ka c+\delta,\,\ka
c+\eps)$ for some arbitrarily small $\delta\in (0,\,\eps)$, and assume that
$\eps$ is not in a certain zero measure set, to comply with the
official definition of the filtered contact or Floer homology
requiring the end points of the interval to be outside the action
spectrum.  Furthermore, in \eqref{eq:hom}, we in fact have
$\HC_{d_\ka+1}^{(\ka c,\,\ka c+\eps)}(M,\alpha)\cong \Q$ as in the
Hamiltonian case.
\end{Remark}

Theorem \ref{thm:sdm} follows immediately from Theorem
\ref{thm:ch-sdm}.

\begin{Remark}
  Note that in Theorem \ref{thm:sdm} we make no claim concerning the
  growth rate of the number of closed Reeb orbits. Recall in this
  connection that, in the settings of the Conley conjecture, for any
  Hamiltonian diffeomorphism with isolated fixed points,
  every sufficiently large prime occurs as a period of a simple
  periodic orbit; \cite{Gi:CC,GG:gaps,He:CC,Hi,SZ}. In a similar vein,
  the number of closed geodesics on $S^2$ with length less than or
  equal to $\ell$ grows at least as prime numbers, i.e., it
  is bounded from below by $\const\cdot\ell/\log\ell$; \cite{Hi93}.
\end{Remark}

\begin{Remark}
\label{rmk:min2}
An assertion similar to Theorem \ref{thm:ch-sdm} holds when an SDM
orbit is replaced by an SDMin Reeb orbit. In this case we have
$\HC_{2m\ka-3}^{(\ka c-\eps,\,\ka c)}(M,\alpha)\neq 0$. As has been
mentioned in Remark \ref{rmk:min0}, this result partially generalizes
the main theorem of \cite{Hi97} asserting, in the presence of SDMin
geodesic, the existence of infinitely many prime closed geodesics on
$S^2$ with length in certain intervals.
\end{Remark}

\subsection{Proof of Theorem \ref{thm:ch-sdm}}
\subsubsection{Local model and its extension} Let a simple Reeb orbit
$x$ be an SDM of the Reeb flow of $\alpha$ on $M^{2n-1}$ with respect
to a trivialization of $\ker\alpha\!\mid_x$.  There exists a tubular
neighborhood $U=B\times S^1$ (where $B$ is the ball $B_R$ of radius
$R$ in $\R^{2n-2}$), adapted to the trivialization in the obvious
sense, and a Hamiltonian $H\colon B\times S^1\to (0,\,\infty)$ such
that
\begin{itemize}
\item[(H1)] $\alpha=\lambda+Hdt$, where $\lambda=(p\,dq-q\,dp)/2$ is the
  $\U(n-1)$-invariant primitive of the standard symplectic structure
  on $B$;
\item[(H2)] $H_t(0)=c$ is a strict local maximum of $H_t$ for all $t\in
  S^1$, and $c$ is independent of $t$;
\item[(H3)] $\d^2H_t$ at $0$ has zero eigenvalues.
\end{itemize}
This fact readily follows from Proposition \ref{prop:sdm2}; see the
proof of \cite[Lemma 5.2]{HM}. Note that for any $\eta_0>0$ this can
be done so that $\|\d^2 H_t\|<\eta_0$ on $B$ for all $t\in
S^1$. (Recall also that the eigenvalues of a quadratic form on a
symplectic vector space are, by definition, the eigenvalues of the
linear symplectic vector field it generates; see Example
\ref{exam:sdm}.)

Next, as in \cite{Gi:CC,GG:gaps}, let us consider functions
$H_\pm\colon B\to (0,\infty)$ such that 
\begin{equation}
\label{eq:H_pm}
H_+\geq H_t\geq H_- \textrm{ for all $t\in S^1$}, 
\end{equation}
with equality attained only at $0\in B$, and $H_+\equiv c_+>0$ and
$H_-\equiv c_->0$ on the region $Y=B_R\setminus B_{R/2}$ near the
boundary of $B$. These functions are also required to be $C^0$-close
to $H$ and satisfy some additional conditions.

Namely, set $\rho=\|z\|$ on $B$.  Then $H_+$ is a
bump--function on $B$, constant near $0$ and on $Y$, and depending only on
$\rho$. In other words, we have

\begin{itemize}
\item[(H$_{+}$)] $H_+$ is a decreasing function of
  $\rho$, equal to $c$ for $\rho$ close to $0$ and equal to $c_+>0$ for 
$\rho\in [R/2,\, R]$. 
\end{itemize}

The function $H_-$ is obtained by composing a certain function $F$ on
$B$ with a (linear) symplectic transformation $\Psi\colon
\R^{2n-2}\to\R^{2n-2}$, i.e., $H_-=F\circ \Psi$. The function $F$ is
also a bump function on $B$ depending only on $\rho$, but it has a
non-degenerate maximum at $0\in B$. To be more specific, we have

\begin{itemize}
\item[(H$_{-}$1)] $F$ is a decreasing function of
  $\rho$, equal to $c$ at $0$ and equal to $c_->0$ on~$Y$;

\item[(H$_{-}$2)] $F$ has a non-degenerate maximum at $0$ and the
  eigenvalues of the quadratic form $\d^2 F(0)$ on the symplectic
  vector space $T_0 B$ are small;

\item[(H$_{-}$3)] there exists a family of linear symplectic
  transformations $\Psi_s$, $s\in [0,\,1]$, with $\Psi_0=\id$ and
  $\Psi_1=\Psi$ such that $F_s:=F\circ \Psi_s\leq H_+$ and $F_s\equiv
  c_-$ on~$Y$ for all $s$.

\end{itemize}
The existence of functions $H_\pm$ with these properties is
guaranteed by (H2) and (H3). (Note that the main point of (H$_-$2) is
that for any $\eta_1>0$ one can construct $F$ such that the eigenvalues
of $d^2F(0)$ are smaller than $\eta_1$.) In addition, the functions
$H_+$ and $F$ are required to satisfy some other natural conditions such
as those for the bump functions in \cite[Section 7.2]{Gi:CC}. These
conditions are not explicitly used in the argument below.

\begin{Remark}
  It is worth pointing out that this construction of $H_-$ differs
  slightly from the one used in \cite{Gi:CC}, where $H_-$ is obtained
  from $F$ by applying a loop of Hamiltonian diffeomorphisms to the
  flow of $F$ rather than composing $F$ with $\Psi$. Both, here and in
  the Hamiltonian setting, either construction can be utilized, for
  only the existence of an isospectral deformation from $F$ to $H_-$
  majorated by $H_+$ is essential; see (H$_-3$). However, we find our
  present approach somewhat simpler than the original one. Furthermore,
  in both settings, it would be convenient to replace the condition
  that $H\equiv c$ near zero in (H$_+$) by that $\d^2 H(0)=0$.  (Then,
  for instance, $H_+$ could be taken independent of $\eps$.)  However,
  we keep (H$_+$) in its present form to simplify references to the
  arguments in \cite{Gi:CC,GG:capacity,GG:gaps}.
\end{Remark}

Consider now the symplectic manifold $\Pi=U\times L=B\times S^1\times
L$ with the symplectic form $\hat{\omega}=\omega+dh\wedge dt$, where
$L\subset \R$ is some interval containing $c$, and $h$ is the
coordinate on $L$. We can choose $L$ to be small and such that $\Pi$
contains the graphs of $H_\pm$, $F$ and $H$. (Recall that the
functions $H_\pm$ are $C^0$-close to $H$, and clearly $H$ has small
variation on $U$.) There exists a symplectic embedding of $\Pi$ into
the symplectization $\hat{M}=M\times (0,\,\infty)\subset T^*M$ sending
the graph of $H$ to $U$ and such that the pull-back of $\tau\alpha$,
where $\tau$ is the coordinate on $(0,\,\infty)$, is the form
$\lambda+hdt$. (Here we have identified $M$ and the graph of $\alpha$
in $\hat{M}$ and $M\times \{1\}\subset \hat{M}$.) From now on we treat
$\Pi$ as a subset of $\hat{M}$. It is important however to keep in
mind that such an embedding does not in general send the fibers $L$ of
the projection $\Pi\to U$ to the fibers $(0,\,\infty)$ of the
projection $\pi\colon \hat{M}\to M$.

Let $U_\pm$ be the images of the graphs of $H_\pm$ in
$\hat{M}$. These images are contact submanifolds in $\hat{M}$ (with the
restriction of $\tau\alpha $ taken as a contact form) lying, in the
obvious sense, above and below $U$.  Furthermore, the contact
submanifolds $U_\pm$ in $\hat{M}$ extend to contact submanifolds
$M_\pm$, contactomorphic to $(M,\ker\alpha)$ and lying again above and
below $M$ in $\hat{M}$. We use the notation $\alpha_\pm$ for the
resulting contact forms obtained by restricting $\tau\alpha$ to
$M_\pm$.

In other words, we have constructed two contact embeddings $M_\pm$ of
$M$ into $\hat{M}$, lying above and below $M$ and $C^0$-close to
$M$. The pull-backs $\alpha_\pm$ of $\tau\alpha$ to $M_\pm$ are
contact forms such that $\alpha_\pm =f_\pm\alpha$, where $f_+>1$ and
$f_-<1$, on the complement of, say, $B_{R/2}\times S^1$. Moreover, it
is not hard to show that the contact structures $\ker\alpha_\pm$ are
isotopic to $\ker\alpha$ with support in $U$.

\begin{Remark}
  For a suitable embedding $\Pi\hookrightarrow \hat{M}$, it is not
  hard to guarantee that $\alpha_+$ corresponds to the differential
  form $f_+ \alpha\!\mid_U$ with $f_+\geq 1$. Whether the image of
  $\alpha_-$ corresponds to a differential form or not, (i.e., if it
  is transverse to the fibers of $\pi$) is less clear. (This would be
  the case if $H_-$ was rotationally symmetric, i.e., when $H_-=F$.)
  However, the image of $\alpha_-$ corresponds to the differential
  form $f_-\alpha$ with $f_-<1$ near the boundary of $U$.)
\end{Remark}

Finally, the part of $\hat{M}$ containing $M$ and $M_\pm$ can be
symplectically embedded into the completion $\hat{W}=W\cup M\times
[1,\infty)$. As a consequence, we obtain exact symplectic fillings
$W_\pm$ of $(M,\alpha_\pm)$. Furthermore, the region $V\subset
\hat{M}$ bounded by $M_-$ and $M_+$ can be treated as a symplectic
cobordism from $M_-$ (the negative end) to $M_+$ (the positive
end). (Strictly speaking $M_+$ and $M_-$ intersect along $x$. For the
sake of brevity, we ignore this issue; for the intersection can be
eliminated by an arbitrarily small perturbation of $M_-$, which, in
turn, would only have an arbitrarily small effect on the action
filtration.) Clearly, $V$ factors as a composition of cobordisms from
$(M,\alpha_-)$ to $(M,\alpha)$ and from $(M,\alpha)$ to
$(M,\alpha_+)$.

Thus, for any (generic) action interval $I$, we obtain the
maps
$$
\HC_*^I(M,\alpha_+)\to \HC_*^I(M,\alpha)\to \HC_*^I(M,\alpha_-),
$$
where the contact homology of $(M,\alpha_\pm)$ is taken with respect
to the filling $W_\pm$ introduced above. To prove the theorem, it now
suffices to show that the map
\begin{equation}
\label{eq:map1}
\HC_*^{(\ka c,\,\ka c+\eps)} (M,\alpha_+)\to \HC_*^{(\ka c,\,\ka
  c+\eps)} (M,\alpha_-)
\end{equation}
is non-zero in degree $d_\ka+1$. 

Before turning to the proof of \eqref{eq:map1}, let us elaborate on
the inter-dependence of various ingredients of the above
construction. The constants $c_\pm$, depending only on $U$ and
$H$, are fixed first. The upper limit of $\eps$ and the lower bound
$\ka_0(\eps)$ for $\ka$ depend on $U$, $H$ and these constants. Then
we chose the function $H_+$, which depends on $H$ and $U$ and $\eps$
and, of course, $c_+$ . Finally, the function $H_-$ is then chosen
individually for each $\ka$. Specifics of extending $\alpha_\pm$
beyond $\Pi$ are immaterial.

\subsubsection{Direct sum decomposition of contact homology}
\label{sec:direct_sum}
Our goal in this section is to describe a contact analogue of the
direct sum decomposition of filtered Floer homology from \cite[Section
5.1]{GG:gaps}.

Consider the class of contact forms $\beta$ on $M$ satisfying the
following conditions:
\begin{itemize}
\item[(B1)] $\beta=\lambda+Kdt$ on $U$, where $K>0$ is a function on
  $U$ taking values in $L$, and $K$ is constant
  on $Y\times S^1$;
\item[(B2)] $\beta=f\alpha$ outside $U$.
\end{itemize}

\begin{Remark}
  The condition (B2) is not essential, but it is convenient to have it in
  what follows. Note that, as a consequence, $(M,\beta)$ is fillable
  by $W$ with a symplectic structure modified near the boundary. A
  variant of such a fillability requirement would be sufficient for
  our purposes. The assumption that $K$ takes values in $L$ can also
  be relaxed.
\end{Remark}

Among the forms in this class are $\alpha_\pm$ and the form
$\alpha_F$ which is equal to $\alpha_-$ outside $U$ and to
$\lambda+Fdt$ in $U$.

Set $a:=K\!\mid_{Y\times S^1}$.  Let $I$ be a interval, not containing
any of the points $\{k a\mid k\in \N\}$. We
require this interval to be sufficiently short: $|I|< \eps_U$, where
$|I|$ stands for the length of $I$ and the value of $\eps_U$ is to be
specified shortly.  Fix $\ka$, and let $\gamma_\ka$ be the free
homology class of $\ka S^1\subset U$ in $W$. (Note that $\gamma_\ka=0$
when $S^1$ is contractible in $W$.)

Even though we are working with the linearized contact homology, we can
view the filtered contact homology complex of $\beta$ as generated by
the closed Reeb orbits of a small non-degenerate perturbation of
$\beta$. Thus let us first assume that all closed Reeb orbits of
$\beta$ except those in $Y\times S^1$ are non-degenerate. (It would be
sufficient to require this only for the orbits with action in $I$ and
in the class $\gamma_\ka$.)  Let finally $\CC_*^I(U,\beta,\ka[S^1])$
be the graded subspace of the complex $\CC_*^I(\beta, \gamma_\ka)$
generated by the orbits lying in $U$ in the homotopy class $\ka [S^1]$. 

\begin{Lemma}
\label{lemma:direct-sum}
There exists a constant $\eps_U>0$, depending only on $U$ and
$\lambda$ (or rather on $\lambda\!\mid_{Y\times S^1}$) such that
$\CC_*^I(U,\beta,\ka[S^1])$ is a subcomplex and, in fact, a direct
summand in the complex $\CC_*^I(\beta, \gamma_\ka)$ whenever
$|I|<\eps_U$. In other words, we have the direct sum decomposition
$$
\CC_*^I(\beta,\gamma_\ka)=\CC_*^I(U,\beta,\ka[S^1])\oplus 
\CC_*^I(M\setminus U,\beta,\gamma_\ka)
$$
for some subcomplex $\CC_*^I(M\setminus U,\beta,\gamma_\ka)$. As a
consequence, we also obtain a direct sum decomposition
\begin{equation}
\label{eq:direct-sum}
\HC_*^I(\beta,\gamma_\ka)=\HC_*^I(U,\beta,\ka[S^1])\oplus 
\HC_*^I(M\setminus U,\beta,\gamma_\ka)
\end{equation}
regardless of whether $\beta$ is non-degenerate
or not, provided that the end points of $I$ are
outside the action spectrum of $\beta$, the interval $I$ does not
contain any of the points $\{ka\mid k\in\N\}$, and $|I|<\eps_U$.
\end{Lemma}

\begin{Remark}
By definition, the complex $\CC_*^I(M\setminus U,\beta,\gamma_\ka)$
is generated by the orbits that are not in $\CC_*^I(U,\beta,
\ka[S^1])$. Note that in spite of the notation these orbits can be
contained in $U$. The exact nature of this complex and of its homology is
inessential for our purposes. 
\end{Remark}

This lemma is a contact analogue of \cite[Lemma 5.1]{GG:gaps} and it
can be proved exactly in the same fashion. (See also \cite{He:CC}.)
The key observation is that for any holomorphic curve $u$ in $\hat{M}$
its image in $M$ further projects to a holomorphic curve in
$Y$, i.e., the map $\pi(u)\cap (Y\times S^1) \to
Y$ is holomorphic. Since crossing the shell $Y$ requires for a
holomorphic curve to have some positive minimal energy $\eps_U$, the
lemma follows.

The important feature of the group $\HC_*^I(U,\beta,\ka[S^1])$ is
that it can be evaluated by entirely local means, by working only
with closed Reeb orbits in $U$ and holomorphic cylinders in the
symplectization of $U$.

Let now $\beta_0$ and $\beta_1$ be two forms in the above class.
Assume that these forms bound a symplectic cobordism from $\beta_1$
(the negative end) to $\beta_0$ (the positive end) in $\hat{M}$. To be
more specific, let us require (although these conditions can be
considerably relaxed) that $\beta_{i}=f_i\alpha$ outside $U$ with
$i=0,1$ and $f_0\geq f_1$ and that $\beta_i=\lambda+K_i dt$ in $U$
with $K_0\geq K_1$. Then the decomposition \eqref{eq:direct-sum} is
preserved by the natural map from $\HC_*^I(\beta_0,\gamma_\ka)$ to
$\HC_*^I(\beta_1,\gamma_\ka)$, provided that $I$ does not meet the
sets $\{k a_0\mid k\in \N\}$ and $\{k a_1\mid k\in \N\}$, where
$a_i=K_i\mid_{Y\times S^1}$. In other words, we have the commutative
diagram:
$$
\begin{CD}
\HC_*^I(\beta_0,\gamma_\ka)@. \,= \, @. \HC_*^I(U,\beta_0,
\ka[S^1])@. \, \oplus \, @. 
\HC_*^I(M\setminus U,\beta_0,\gamma_\ka)\\
@VVV @. @VVV @. @VVV\\
\HC_*^I(\beta_1,\gamma_\ka)@. \,= \, @. \HC_*^I(U,\beta_1,\ka[S^1])@. \,\oplus \,@. 
\HC_*^I(M\setminus U,\beta_1,\gamma_\ka)
\end{CD} 
$$

Applying this to $\alpha_+=\beta_0$ and $\alpha_-=\beta_1$ with
$I=(\ka c,\ka c+\eps)$, we see that it suffices to show that the map
\begin{equation}
\label{eq:map2}
\HC_{d_\ka+1}^{(\ka c,\,\ka c+\eps)} (U,\alpha_+,\ka[S^1])\to \HC_{d_\ka+1}^{(\ka c,\,\ka
  c+\eps)} (U,\alpha_-,\ka[S^1])
\end{equation}
is non-zero when all of the above requirements are satisfied. Note
that the map \eqref{eq:map2} is completely independent of the behavior
of $\alpha_\pm$ outside $U$ or of the topology of $(M,\ker\alpha)$,
which is why the particulars of extending $\alpha_\pm$ beyond $U$ are not
essential for the proof. We will examine this map more closely in
the next section.

\begin{Remark}
\label{rmk:ch-vs-fh}
  By the very definition, the complex $\CC_*(U,\beta,\ka[S^1])$ is
  generated by the $\ka$-periodic orbits of $K$ (with the exception of
  the bad orbits). Hence, the first step in
  calculating the homology $\HC_*(U,\beta,\ka[S^1])$ is that of
  analyzing the periodic orbits of $K$, which is a problem in
  Hamiltonian dynamics; see Section \ref{sec:homology}. Taking this
  point one step further, one can show that
  $\HC_*^{I}(U,\beta,\ka[S^1])=\HF_*^{\Z_\ka,I}(K)$, where the right
  hand side is the $\Z_\ka$-equivariant filtered Floer homology of $K$
  on $B$, which is defined as long as $|I|<\eps_U$. The proof of
  this fact is completely analogous to the proof of Theorem
  \ref{thm:ch-vs-fh}. This consideration applies to the cobordism maps
  as well.
\end{Remark}

\begin{Remark}
  The condition that the interval $I$ does not intersect the set
  $\{ka\mid k\in\N\}$ is purely technical and can be eliminated by
  slightly modifying our requirements on the form $\beta$.  This can
  be done by constructing a contact analogue of the direct sum
  decomposition of Floer homology from \cite{He:CC}; cf.\ \cite{Us}.
\end{Remark}

\subsubsection{Homological calculation} 
\label{sec:homology}
Our goal is to show that the map \eqref{eq:map2} is non-zero, which
implies the assertion of the theorem. In this section, we will prove
that, for a suitable choice of the parameters of the construction,
this map is in fact an isomorphism from $\Q$ to $\Q$:
\begin{equation}
\label{eq:map3}
\Q=\HC_{d_\ka+1}^{(\ka c,\,\ka c+\eps)}
(U,\alpha_+,\ka[S^1])\stackrel{\cong}{\longrightarrow} \HC_{d_\ka+1}^{(\ka c,\,\ka
c+\eps)} (U,\alpha_-,\ka[S^1])=\Q.
\end{equation}

The proof of \eqref{eq:map3} is based, in particular, on the following
result applied to the contact homology localized to $U$ as above. 

\begin{Proposition}[Invariance of the Filtered Contact Homology]
\label{Prop:invariance}
  Let $\beta_s$, $s\in [0,\,1]$, be a family of contact forms on $M$
  and $I_s$ a family of intervals such that for every $s$ the end
  points of $I_s$ are outside the action spectrum of $\beta_s$. Then
  the contact homology spaces $\HC_*^{I_s}(\beta_s)$ are
  isomorphic. Furthermore, assume that the forms $\beta_s$ foliate, in
  the obvious sense, a cylindrical cobordism from $\beta_1$ to
  $\beta_0$ and the interval $I_s=I$ is fixed. Then the natural
  map $\HC_*^I(\beta_0)\to \HC_*^I(\beta_1)$ is an
  isomorphism. Finally, to have an isomorphism in a particular
  degree $*=k$, it is sufficient to assume only that the end points of $I_s$
  are outside the action spectra corresponding to the
  orbits of degree $k$ and $k\pm 1$.
\end{Proposition}

This proposition can be proved in exactly the same way as its
Floer homological counterparts; see, e.g., \cite{BPS,Gi:coiso,Vi}.

As the first application, by (H$_-$3), we have the following commutative diagram
\begin{equation}
\label{eq:diag-iso}
\xymatrix{
{\HC_{*}^{I}
(U,\alpha_+,\ka[S^1])} \ar[d] \ar[rd] &\\
{\HC_{*}^{I} (U,\alpha_-,\ka[S^1])} \ar[r]^{\cong} & {\HC_{*}^{I} (U,\alpha_F,\ka[S^1])}
}
\end{equation}
where the form $\alpha_F$ is defined similarly to $\alpha_-$, but with
$F$ in place of $H_-$; see Section \ref{sec:direct_sum}. This is a
contact analogue of, say, the results in \cite[Sections 2.2.1 and
5.3.2]{GG:gaps}. As a consequence of \eqref{eq:diag-iso}, we may
replace $\alpha_-$ by $\alpha_F$ in \eqref{eq:map3}. Thus, to complete
the proof of the theorem, we need to establish the isomorphism
\begin{equation}
\label{eq:map4}
\Q=\HC_{d_\ka+1}^{(\ka c,\,\ka c+\eps)}
(U,\alpha_+,\ka[S^1])\stackrel{\cong}{\longrightarrow} \HC_{d_\ka+1}^{(\ka c,\,\ka
c+\eps)} (U,\alpha_F,\ka[S^1])=\Q.
\end{equation}

Let us now specify the parameters of the construction, needed to ensure
that \eqref{eq:map4} holds. As a starting point, the form $\alpha$ is
given and the set $\Pi$ is fixed in advance, as are 
the function $H$ satisfying the requirements (H1-H3) and the interval
$L$. Next we pick $c_\pm$. Here, in addition to $c_\pm\in L$, we require the
points $c_\pm$ to be relatively close to $c$, e.g.,
$|c-c_\pm|<c/2$. It is also convenient to take these points from the
set $\Q c$. Thus $c_\pm =p_\pm c/ q_\pm$ for some integers
$p_\pm$ and $q_\pm$.

As the next step, we determine the upper bound on
$\eps>0$. Namely, we have $\eps<\eps_U$ and $\eps$
is small relative to $|c-c_\pm|$ (e.g., $\eps<|c-c_\pm|/2$) and,
finally, $\eps<c/\max\{q_-,q_+\}$.  It is easy to see that under the
latter requirement, for any sufficiently large $\ka$, the
interval $I=(\ka c,\, \ka c+\eps)$ does not intersect the sets $\{k
c_\pm \mid k\in\N\}$. (Without the assumption that $\eps<c/\max\{q_-,q_+\}$,
this would be true only for an infinite sequence of iterations $\ka$.)
Hence the machinery from Section \ref{sec:direct_sum} applies. We pick
$\ka$ meeting the latter requirement and such that
\begin{equation}
\label{eq:ka}
\ka |c-c_\pm|>\pi R^2.
\end{equation}

The function $H_+$ (with $c_+$ as above) is required to meet the
condition (H$_+$) and be such that $\pi r^2<\eps$, where $r$ is the
radius of the ball on which $H\equiv c$. Note that this function can
be chosen independently of $\ka$.

Finally, we fix $H_-$ (with $c_-$ as above) satisfying
(H$_-$1-H$_-$3). Here, (H$_-$2) reads specifically as that the
eigenvalues of $\d^2 F(0)$ are so small (while still positive) that
the Conley--Zehnder index of $\ka\, \d^2 F(0)$ is $n-1$, i.e., all
eigenvalues of the latter quadratic form are smaller than $\pi$.

With these conditions in mind, there are several (ultimately not so
different) ways to establish \eqref{eq:map4} and thus finish the proof
of the theorem.

First of all, note that the fact that both of the homology groups in
\eqref{eq:map4} are indeed $\Q$ (under the assumptions (H$_-$2) and
\eqref{eq:ka}) can be proved by simple analysis of the $\ka$-periodic
orbits of $H_+$ and $F$, which is done in, e.g.,
\cite{Gi:CC,GG:capacity}. (Here the picture becomes particularly
transparent once the Morse--Bott approach to contact homology is
utilized; see \cite{Bou} and references therein. Note also that the
action of $\alpha_+$ on the corresponding orbit is approximately equal
to $\ka c+\pi r^2$ for large $\ka$ and $H_+$ constructed exactly as in
\cite[Section 7.3]{Gi:CC}.) Then to prove that the map \eqref{eq:map4}
is an isomorphism, it suffices to analyze the behavior of
$\ka$-periodic orbits for a suitable homotopy from $H_+$ to $F$ and
apply Proposition \ref{Prop:invariance}, cf.\ \cite{GG:capacity}.

Alternatively, one can argue as follows (cf.\ \cite{GG:gaps}). Let us
fix a sufficiently small parameter $\delta>0$ and suppressing $U$ and
$\ka[S^1]$, which are fixed now, in the notation let us write
$\HC_*^I(\beta)$ for $\HC_*^I(U,\beta,\ka[S^1])$. (Note that $\delta$
depends on $H_+$, $F$, $\eps$, and $\ka$. In particular, $\delta\ll
\eps$.)  Then, as is easy to see (cf.\ \cite{GG:capacity}),
$$
\HC_{d_\ka}^{(\ka c-\delta,\,\ka c+\delta)} (\alpha_+)=\Q=
\HC_{d_\ka}^{(\ka c-\delta,\,\ka c+\delta)} (\alpha_F)
$$
and, moreover, the natural continuation map between these spaces is an
isomorphism.  Here the main point is that the orbit $x^\ka$ is the
only generator of degree $d_\ka$ for the contact complexes of
$\alpha_+$ and $\alpha_F$, where for $\alpha_+$ we slightly perturb
$H_+$ to make $0\in B$ into a non-degenerate maximum; cf.\ \cite[Lemma
2.5]{Gi:CC}.

Next, focusing first on $\alpha_+$, we have
$$
\HC_{d_\ka+1}^{(\ka c-\delta,\,\ka c+\eps)} (\alpha_+)=0
\textrm{ and }
\HC_{d_\ka}^{(\ka c-\delta,\,\ka c+\eps)}
(\alpha_+)=0 .
$$
These identities are proved, for instance, by deforming
$H_+$ to a Hamiltonian without periodic orbits of required index and
applying Proposition  \ref{Prop:invariance}. Thus, from the long exact
sequence
$$
\cdots \to \HC_{*}^{(\ka c-\delta,\,\ka c+\delta)} (\alpha_+)
\to
\HC_{*}^{(\ka c-\delta,\,\ka c+\eps)} (\alpha_+)
\to
\HC_{*}^{(\ka c+\delta,\,\ka c+\eps)} (\alpha_+)
\to\cdots,
$$
we conclude that the connecting map
$$
\HC_{d_\ka+1}^{(\ka c+\delta,\,\ka c+\eps)} (\alpha_+)
\stackrel{\cong}{\longrightarrow}\HC_{d_\ka}^{(\ka c-\delta,\,\ka
  c+\delta)} (\alpha_+)
=\Q
$$
is an isomorphism. 

This argument applies to $\alpha_F$ word-for-word and we also have the
isomorphism
$$
\HC_{d_\ka+1}^{(\ka c+\delta,\,\ka c+\eps)} (\alpha_F)
\stackrel{\cong}{\longrightarrow}\HC_{d_\ka}^{(\ka c-\delta,\,\ka
  c+\delta)} (\alpha_F)
=\Q
$$
Combining these facts, we arrive at the commutative diagram
$$
\begin{CD}
\HC_{d_\ka+1}^{(\ka c+\delta,\,\ka c+\eps)}(\alpha_+)@>\cong>>
\HC_{d_\ka}^{(\ka c-\delta,\,\ka c+\delta)}(\alpha_+)
@.\,=\,\Q\\
@VVV @VV{\cong}V @.\\
\HC_{d_\ka+1}^{(\ka c+\delta,\,\ka c+\eps)}(\alpha_F)@>\cong>>
\HC_{d_\ka}^{(\ka c-\delta,\,\ka c+\delta)}(\alpha_F)
@.\,=\,\Q\\
\end{CD} 
$$
where the horizontal maps and the second vertical map are
isomorphisms. Hence the first vertical map is also an isomorphism.
This (together with Remark \ref{rmk:interval}) implies~\eqref{eq:map4}.

Finally, one may take advantage of the isomorphism between
$\HC_*^{I}(U,\beta,\ka[S^1])$ and $\HF_*^{\Z_\ka,I}(K)$ from Remark
\ref{rmk:ch-vs-fh}. It is not hard to see that for the Hamiltonians
$H_\pm$ and $F$ the $\Z_\ka$-action on the Floer complex is trivial,
and hence proving \eqref{eq:map3} and \eqref{eq:map4} amounts to
establishing isomorphisms in the ordinary Floer homology, which is done
exactly in this setting in \cite{GG:capacity}.

\section{Appendix: Equivariant and invariant Morse homology}
\label{sec:Morse}
In this section, we illustrate the constructions from Section
\ref{sec:comparison} by analyzing a finite-dimensional 
model in which the arguments are particularly transparent.

Consider a Morse function $f$ on a closed oriented manifold $M$. To
define the Morse complex $\MC_*(f)$ over $\Z$ or $\Q$, we fix a
Riemannian metric such that the stable and unstable manifolds (for the
anti-gradient flow) of the critical points of $f$ intersects
transversely, and assign an orientation to every unstable
manifold. The complex $\MC_*(f)$ is then generated by the critical
points of $f$ and graded by the Morse index. The differential $\p
x=\sum \mu(x,y)y$ is defined by counting anti-gradient trajectories
connecting a point $x$ of index $k$ and with points $y$ of index $k-1$
with signs depending on whether the intersection orientation of a
trajectory (induced by the orientation of $M$ and those of the
stable/unstable manifolds) matches its orientation as a flow line. It
is important that the differential depends on the metric even though
the graded vector space $\MC_*(f)$ is completely determined by the
function $f$ only. Note also that the transversality requirement,
often referred to as the Morse--Smale condition, is satisfied for a
generic metric. The homology $\HM_*(M)$ of the complex $(\MC_*(f),\p)$
is isomorphic to $\H_{*}(M)$ and called the Morse homology of
$f$. (See, e.g., \cite[Chap.\ 7]{Jo} for more details.)

Let now $G$ be a finite group acting on $M$ by orientation preserving
diffeomorphisms. Assume that $f$ is $G$-invariant. Then the
equivariant Morse complex $\MC_*^G(f)$ of $f$ and the equivariant
Morse homology $\HM_*^G(f)$ are defined as follows.  Fix a sequence of
finite-dimensional smooth approximations $EG_N\to BG_N$ of the
universal bundle $EG\to BG$. More specifically, this is a sequence of
closed smooth manifolds $EG_N$ with free $G$-action, such that
$\pi_k(EG_N)=0$ for $k=1,\ldots, k_N\to\infty$, and a sequence of
$G$-equivariant embeddings $EG_N\hookrightarrow EG_{N+1}$. We set
$BG_N=EG_N/G$. (Such approximations can be obtained, for instance, by
taking a faithful representation $G\to\U(m)$ and letting $EG_N$ be the
Stiefel manifold of unitary $m$-frames in $\C^N$ with $G$ acting via
the standard $\U(m)$-action. When $G=\Z_\ka$, we have $EG_N=S^{2N-1}$
with $\Z_\ka$ acting diagonally, and $BG_N$ is a lens space.)
Furthermore, pick a sequence of $G$-invariant Morse functions $h_N$ on
$EG_N$ such that $h_{N+1}\!\mid_{EG_N}=h_N$ and the critical points of
$h_N$ are necessarily critical points of $h_{N+1}$ of the same
index. Finally, we also require that all critical points of $h_{N+1}$
outside $EG_N$ have index grater than some $\mu_N\to\infty$. (This is
an extra condition on both $h_N$ and $EG_N$ needed to ensure that the
Morse complexes $\MC_*(f_N)$ of the functions $f_N$ on
the finite-dimensional approximations $M_N\to (EG\times M)/G$ defined below
stabilize as $N\to\infty$.)

The function $f+h_N$ descends to a smooth Morse function $f_N$ on the
quotient $M_N=(EG_N\times M)/G$. (The quotient is smooth since $G$
acts freely on $EG_N\times M$, and we should think of $M_N$ as a
finite-dimensional approximation to $(EG\times M)/G$.) Note that again
$f_{N+1}\!\mid_{M_N}=f_N$, the critical points of $f_N$ are
necessarily critical points of $f_{N+1}$ of the same index, and all
critical points of $f_{N+1}$ outside $M_N$ have index greater than
$\mu_N\to\infty$.  Pick a sequence of Riemannian metrics on $M_N$ such
that $M_N\subset M_{N+1}$ is invariant under the gradient flow of
$f_{N+1}$ and the Morse--Smale condition is satisfied for each
$f_N$. Clearly, there are natural maps of complexes $\MC_*(f_N)\to
\MC_*(f_{N+1})$ and, moreover, these complexes stabilize in every
fixed range of degrees as $N\to \infty$. By definition, the
equivariant Morse complex of $f$ is $\MC_*^G(f):=\varinjlim
\MC_*(f_N)$ and its homology $\HM_*^G(f)=\varinjlim \HM_*(f_N) $ is
the equivariant Morse homology of $f$. Note that this complex does
depend on the approximation scheme $(EG_N,h_N)$, while its homology is
independent of the approximation and isomorphic to the equivariant
homology $\H^G_{*}(M)$ of $M$.

Next let us assume that $M$ admits a $G$-invariant metric which
satisfies the Morse--Smale condition. In contrast with the
non-equivariant case, such metrics need not exist as simple examples
show. Define an action of $G$ on $\MC_*(f)$ by setting $g\in G$ to
send a critical point $x$, a generator of $\MC_*(f)$,
to $\pm g(x)$. Here $g(x)\in M$ is the image of $x$ under the map $g$
and the sign is determined by whether or not $g$ matches the orientations of
the unstable manifold of $x$ and the unstable manifold of $g(x)$. This
choice of signs guarantees that we indeed have a $G$-action on the
complex $\MC_*(f)$, i.e., that the action and $\p$ commute. Hence we
also obtain a $G$-action on $\HM_*(f)$. In particular, we have the
invariant subcomplex $\MC_*(f)^G$ and the invariant part $\HM_*(f)^G$
of homology.

All these constructions extend in a straightforward way to the filtered
Morse homology and to the local Morse homology (see, e.g.,
\cite[Section 3.1]{Gi:CC}).

Up to this point, the choice of the coefficient ring was immaterial:
it could be $\Z$ or any other ring or field. Moreover, our construction of
the equivariant Morse homology would go through for any compact Lie
group $G$. However, from this moment on, it becomes essential that all
the complexes and homology groups are taken over $\Q$, or more
generally over a field of zero characteristic, and that the group $G$
is finite. First, note that under these conditions, $\HM_*(f)^G$ is the
homology of the complex $\MC_*(f)^G$ and that $\HM_*(f)^G$ is
isomorphic to $\H_{*}(M/G)$. More importantly, we have

\begin{Proposition}
\label{prop:Morse}
When the ground field is $\Q$ and $G$ is finite, the
equivariant and invariant Morse homology groups are isomorphic:
$\HM_*^G(f)=\HM_*(f)^G$. The same holds for filtered and local
Morse homology.
\end{Proposition}

\begin{Remark}
  Even for filtered or local Morse homology, this proposition is but a
  very particular case of the identification $\H_*^G(M)=\H_*(M/G)$,
  which holds over $\Q$ for any action with finite stabilizers of a
  compact Lie group $G$ on any finite CW-complex. (We refer the reader
  to, e.g., \cite[Appendix C]{GGK} and references therein for this and
  other standard results on equivariant (co)homology used in this
  section.) However, what is important for us is that Proposition
  \ref{prop:Morse} is essentially Morse-theoretic, and hence can be
  translated to the realm of Floer homology; cf.\ the proof of
  Theorem~\ref{thm:ch-vs-fh}.
\end{Remark}

The following simple example illustrates the role of signs in the
construction of the invariant Morse complex.

\begin{Example}[Signs]
  Let $f$ be a hyperbolic quadratic form on $M=\R^2$ and let $G=\Z_2$
  act by central symmetry. Clearly, $f$ is a Morse function with only
  one critical point, the origin $0$. We denote by $x$ the
  corresponding generator of the local Morse homology
  $\MC_*(f,0)$. Thus $\MC_*(f,0)$ and $\HM_*(f,0)$ are both equal to
  $\Q$ and concentrated in degree zero. The central symmetry inverses
  the orientation of the unstable manifold of $0$, and hence sends $x$
  to $-x$. As a consequence, $\MC_*(f,0)^G=0$ and
  $\HM_*(f,0)^G=0$. Likewise, it is not hard to see (for instance,
  using the Morse--Bott construction of the equivariant Morse
  homology, cf.\ \cite[Lecture 3]{Bo}) that $\HM_*^G(f,0)=0$. Indeed,
  $\HM_*^G(f,0)$ is isomorphic to the homology of $\RP^{\infty}$ with
  twisted rational coefficients corresponding to the double cover of
  $\RP^\infty$. This homology is zero.
\end{Example}

\begin{proof}[Proof of Proposition \ref{prop:Morse}]
We will prove the proposition for the ground field $\C$. Then, as is
easy to show, the result for $\Q$ follows.

Under our assumptions on $f$, we can equip $EG_N\times M$ with a
product metric satisfying the Morse--Smale condition. 
Then $\MC_*(f+h_N)=\MC_*(f)\otimes \MC_*(h_N)$, and hence
$$
\MC_*^G(f)=\big(\MC_*(f)\otimes \MC_*(h)\big)^G, 
$$
where we set $\MC_*(h):=\varinjlim \MC_*(h_N)$.

Let us decompose the complexes $\MC_*(f)$ and $\MC_*(h)$, viewed as
representations of $G$, into the sum of isotypical components:
$$
\MC_*(f)=\MC_*(f)^G\oplus \bigoplus_{\sigma} V_\sigma
\textrm{ and }
\MC_*(h)=\MC_*(h)^G\oplus \bigoplus_{\eta} W_\eta,
$$
where direct sums range over all non-trivial irreducible complex
representations $\sigma$ and $\eta$ of $G$. It readily follows from
Schur's lemma that these are indeed direct sums of complexes.

We have
\begin{equation}
\label{eq:hom-decomposition}
\big(\MC_*(f)\otimes \MC_*(h)\big)^G= \big(\MC_*(f)^G\otimes \MC_*(h)^G\big)
\oplus\bigoplus_{\sigma,\eta} \big(V_\sigma\otimes W_\eta\big)^G.
\end{equation}
For a product of an isotypical component corresponding to the trivial
representation and the one corresponding to $\sigma$ or $\eta$ (e.g.,
$\MC_*(f)^G\otimes W_\eta$) is again $\sigma$- or $\eta$-isotypical, and
hence its invariant part is zero. 

The homology of the complex $\MC_*(h)^G$ is the group homology
$\H_*(G)$ and the complex $\MC_*(h)$ is acyclic (except degree
zero). Since $G$ is finite, $\H_*(G;\C)=\C$, concentrated in degree
zero. It follows that all complexes $W_\eta$ are acyclic. Thus
only the first term in \eqref{eq:hom-decomposition} makes a
non-trivial contribution to the homology and $\HM_*^G(f)=\HM_*(f)^G\otimes
\C=\HM_*(f)^G$ as required.
\end{proof}

\begin{Remark}[Multi-valued perturbations] 
\label{rmk:multi-valued2}
A generalization of the proposition to the case where there is no
$G$-invariant metric satisfying the Morse--Smale condition would have
to rely on a variant of the machinery of multi-valued
perturbations. (See \cite{FO,FOOO,LT,HWZ3,HWZ4} and references
therein, and also Remark \ref{rmk:multi-valued}.) This machinery
should enable one to construct a version of the Morse complex with a
natural $G$-action even when a $G$-invariant Morse--Smale metric does
not exist. There is a good conceptual understanding of how this could
be done, but this is still a non-trivial task and, to the best of our
knowledge, there are no published accounts of such a
construction. (Once this is accomplished, the proof of Proposition
\ref{prop:Morse} should go through word-for-word.) This problem can be
thought of as a Morse theoretic analogue of the transversality problem
arising in the definition of the contact homology, cf.\ Section
\ref{sec:comparison} and Theorem \ref{thm:ch-vs-fh}.
\end{Remark}

\end{document}